\documentclass{article}

\usepackage[utf8]{inputenc}
\usepackage[T1]{fontenc}
\usepackage{mathptmx}
\usepackage{mathtools}
\usepackage{hyperref}
\usepackage{stmaryrd}
\usepackage{amsmath}
\usepackage{amssymb}
\usepackage{amsthm}
\usepackage{tikz-cd}%to draw commutative diagrams
\usepackage{geometry}
\usepackage{authblk}%Add mail adress at the end

\newtheorem{prop}{Proposition}

\newtheorem{thm}{Theorem}
\newtheorem{lmm}{Lemma}[section]
\newtheorem*{cor}{Corollary}

\theoremstyle{definition}
\newtheorem{defi}{Definition}

\newtheorem{prob}{Problem}

\newtheorem*{rmk}{Remark}

\newenvironment{prf}{{\noindent\it Proof}\quad}{\hfill $\square$\par}

\title{
On Abel-Jacobi Maps of Lagrangian Families
}
\author{Chenyu Bai}
\date{}

\begin{document}
\setcounter{section}{-1}
\maketitle
\begin{abstract}
 We study in this article the cohomological properties of Lagrangian families on projective hyper-Kähler manifolds. First, we give a criterion for the vanishing of Abel-Jacobi maps of Lagrangian families. Using this criterion, we show that under a natural condition, if the variation of Hodge structures on the degree $1$ cohomomology of the fibers of the Lagrangian family is maximal, its Abel-Jacobi map is trivial. We also construct Lagrangian families on generalized Kummer varieties whose Abel-Jacobi map is not trivial, showing that our criterion is optimal. 
\end{abstract}

\section{Introduction}\label{SectionIntroduction}
Let $X$ be a projective hyper-Kähler manifold~\cite{Beauville}, that is, a simply connected complex projective manifold whose space of holomorphic $2$-forms is generated by a nowhere degenerate $2$-form $\sigma_X$. The dimension of $X$ is an even number $2n$. %Dimension $2$ hyper-Kähler manifolds, the so-called K3 surfaces, being relatively well understood, higher dimensional hyper-Kähler manifolds remain to be well explored. 
A Lagrangian subvariety $L$ of $X$ is a dimension $n$ irreducible possibly singular subvariety of $X$ such that, denoting $j: \Tilde{L}\to L\hookrightarrow X$ a desingularisation of $L$, $j^*\sigma_X=0$ in $H^0(\Tilde{L}, \Omega^2_{\Tilde{L}})$.

In this article, we will be considering Lagrangian families of a hyper-Kähler manifold $X$.
\begin{defi}
A Lagrangian family of a hyper-Kähler manifold $X$ is a diagram
\begin{equation}\label{LagrangianFamily}
    \begin{tikzcd}
    \mathcal L\arrow{r}{q}\arrow{d}{p} & X\\
    B &
    \end{tikzcd}
\end{equation}
where $p$ is flat and projective, $\mathcal L$ and $B$ are connected quasi-projective manifolds and $q$ maps birationally the general fiber $L_b:=p^{-1}(b)$, $b\in B$, to a Lagrangian subvariety of $X$. In what follows, we will denote $j_b$ the composition $L_b\hookrightarrow\mathcal L\to X.$
\end{defi}
Lagrangian families were studied by Voisin in \cite{VoisinTriangle, VoisinLefschetz} in different contexts, as generalizations of Lagrangian fibrations~\cite{Matsushita99}. While Lagrangian fibrations do not exist in general on projective hyper-Kähler manifolds (as this forces the Picard number to be at least $2$), Lagrangian families (and even Lagrangian coverings for which $q$ is dominant) are relatively common. See, for example, constructions in~\cite{IlievManivel, VoisinLag} on Lagrangian families (coverings) on Fano varieties of lines of cubic fourfolds and on Hilbert schemes of K3 surfaces. Works of Voisin~\cite{VoisinTriangle, VoisinLefschetz} indicate that the existence of Lagrangian coverings implies important properties of the hyper-Kähler manifold in question. For example, it is shown in~\cite{VoisinLefschetz} that a very general projective hyper-Kähler fourfold admitting a Lagrangian covering satisfies Lefschetz standard conjecture for degree $2$ cohomology. Studies of examples of Lagrangian families lead us to consider the following problem, which is motivated in the article~\cite{VoisinCoisotrope}.

\begin{prob}\label{Problem1}
Consider a Lagrangian family of a hyper-Kähler manifold $X$ of dimension $2n$ given by a diagram as in  (\ref{LagrangianFamily}). What can be said of the map 
\begin{equation}\label{LagrangianChowMap}
    \begin{array}{cccc}
        \psi_{\mathcal L}: & B & \rightarrow & CH^n(X)  \\
        & b & \mapsto & q_*(L_b)
    \end{array}?
\end{equation}

\end{prob}
Let us first explain the relation between this problem and article~\cite{VoisinCoisotrope}. In~\cite{VoisinCoisotrope}, it is conjectured and proved for some cases that if two \textit{constant cycle} Lagrangian subvarieties (see~\cite{ConstantCycle, VoisinCoisotrope} for definitions) of $X$ have the same cohomological cycle class, then they are rationally equivalent. Note the following
\begin{lmm}
Let $X$ be a hyper-Kähler manifold. Then constant cycle Lagrangian subvarieties of $X$ are rigid as constant cycle subvarieties.
\end{lmm}
\begin{prf}
Following the notations in~\cite{VoisinCoisotrope}, let 
\[S_nX:=\{x\in X: \textrm{ the rational equivalence orbit of } x \textrm{ has dimension} \geq n\}.
\]
As is shown in~\cite[Theorem 1.3]{VoisinCoisotrope}, $S_nX$ is a countable union of irreducible varieties of dimension $\leq n$ and constant cycle Lagrangian subvarieties of $X$ are exactly irreducible components of $S_nX$ of dimension $n$. Hence, constant cycle Lagrangian subvarieties of $X$ are rigid.
\end{prf}
One may now wonder if the condition \emph{constant cycle} can be dropped in the above conjecture. If a hyper-Kähler manifold $X$ has the property that the map $\Psi_{\mathcal L}$ of (\ref{LagrangianChowMap}) is constant for any Lagrangian family,  there are only \emph{finitely many} Chow classes for Lagrangian subvarieties $L$ with a given cohomological cycle class $[L]\in H^{2n}(X,\mathbb Z)$. This is because the Hilbert scheme parametrizing subvarieties with a given cohomological class has only finitely many connected components. Conversely, if $X$ has Lagrangian families with nonconstant map $\Psi_{\mathcal L}$, then the naïve generalization of the above conjecture by dropping \emph{constant cycle} is false.

Many of the known examples give rise to Lagrangian families with constant $\Psi_{\mathcal L}$. The typical ones are $K3$ surfaces and Lagrangian fibrations. Indeed, for a $K3$ surface $S$, the class map $CH^1(S)\to H^2(S,\mathbb Z)$ is injective, so that the fibers of a family, being cohomologically equivalent, are rationally equivalent. For a Lagrangian fibration, the base being rationally connected~\cite{MatsushitaBase, Lin}, the fibers are rationally equivalent. Lagrangian families constructed in~\cite{IlievManivel, VoisinLag} also give a positive answer to this problem since the bases are open subsets of projective spaces.

%Not all Lagrangian families are parametrized by a rationally connected base. In Section~\ref{Examples}, Lagrangian families whose base is not (contained in) a rationally connected variety are constructed in generalised Kummer varieties. It is also shown that these Lagrangian families give negative answers to Problem~\ref{Problem1}.  

%Note that, however, that the base of a Lagrangian family is not necessarily (contained in) a rationally connected variety, which makes this problem more intriguing. We will construct such an example in Section~\ref{SectionExample}.

A weaker invariant of algebraic cycles in a projective manifold is the Abel-Jacobi invariant~\cite[Chapter 12]{Voisin}. If two cycles holomogous to $0$ are rationally equivalent, then they have the same Abel-Jacobi invariant in the intermediate jacobians. Problem~\ref{Problem1} thus motivates the following question. Here we denote $\Phi_X^n: CH^n(X)_{hom}\to J^{2n-1}(X)$ the Abel-Jacobi map.
\begin{prob}\label{Problem2}
Consider a Lagrangian family of a hyper-Kähler manifold $X$ of dimension $2n$ given by a diagram as in  (\ref{LagrangianFamily}). Let $0\in B$ be a point. Under which conditions is the Abel-Jacobi map
\begin{equation}\label{AbelJacobiMap}
    \begin{array}{cccc}
       \Psi_{\mathcal L}^{AJ}: & B  & \rightarrow & J^{2n-1}(X) \\
       & b & \mapsto & \Phi_X^n(q_*(L_b-L_0)) 
    \end{array}
\end{equation}
trivial?
\end{prob}
%In the spirit of the Bloch-Beilinson's conjecture~\cite{Saito, VoisinCitrouille}, there exists conjecturally a finite filtration $F^\bullet CH^n(X)$ on the Chow group of $X$ with rational coefficients such that $F^1CH^n(X)=CH^n(X)_{hom}$ and $F^2CH^n(X)=\ker(\Phi_X^n)$. Hence, Problem~\ref{Problem2} can be viewed as a first step to attack Problem~\ref{Problem1}.

Note that $\Psi_{\mathcal L}^{AJ}(b)=\Phi_{X}^n\circ\Psi_{\mathcal L}(b-0)$, Problem~\ref{Problem2} can be viewed as a first step to study Problem~\ref{Problem1}.

~\newline

In this article, we first give a criterion for the vanishing of the Abel-Jacobi map (\ref{AbelJacobiMap}) for Lagrangian families of a hyper-Kähler manifold (see also Proposition~\ref{Criterion}).
\begin{prop}\label{CriterionIntroduction}
Consider a Lagrangian family of a hyper-Kähler manifold $X$ of dimension $2n$ as in (\ref{LagrangianFamily}), satisfying the following condition  :

\begin{quote}
   $\clubsuit$ For general $b\in B$, the contraction by $q^*\sigma_X$ gives an isomorphism $\lrcorner q^*\sigma_X: T_{B,b}\stackrel{\cong}{\to} H^0(L_b,\Omega_{L_b}).$ 
\end{quote}
 Then the Abel-Jacobi map (\ref{AbelJacobiMap}) is trivial if and only if for general $b\in B$, the restriction map 
\[j_b^*: H^{2n-1}(X,\mathbb Q)\to H^{2n-1}(L_b,\mathbb Q)
\]
is zero.
\end{prop}

The condition $\clubsuit$ is natural. According to~\cite{VoisinLag}, the deformations of a \emph{smooth} Lagrangian subvariety are non-obstructed, and a local deformation is still Lagrangian. Therefore, if we take $(B,b)$ to be a germ of the Hilbert scheme of deformations of a \emph{smooth} Lagrangian subvariety $L\subset X$, and $\mathcal L\to B$ the corresponding family, then condition $\clubsuit$ holds since $T_{B,b}\cong H^0(L_b, N_{L_b/X})$ by unobstructedness and $\lrcorner\sigma_X: H^0(L_b,N_{L_b/X})\to H^0(L_b,\Omega_{L_b})$ is an isomorphism for a smooth Lagrangian variety.  

Using this criterion, we give a response to Problem~\ref{Problem2}.
\begin{thm}\label{MainTheorem}
 Consider a Lagrangian family of a hyper-Kähler manifold $X$ of dimension $2n$ given by a diagram as in (\ref{LagrangianFamily}), satisfying condition $\clubsuit$. Assume that the variation of Hodge structures on the degree $1$ cohomology of the fibers of $p:\mathcal L\to B$ is maximal, i.e., the local period map
\begin{equation}\label{PeriodMap}
    \begin{array}{cccc}
         \mathcal P: & B &\to & Gr(h^{1,0}(L), H^2(L,\mathbb C))\\
          & b & \mapsto & H^{1,0}(L_b)\subset H^1(L_b,\mathbb C)\cong H^1(L,
          \mathbb C),
    \end{array}
\end{equation}
where $L$ is a general fiber of $p: \mathcal L\to B$,
is generically a local immersion. Then the Abel-Jacobi map (\ref{AbelJacobiMap}) is trivial.
\end{thm}

 In the opposite direction, we construct in Section~\ref{Examples} Lagrangian families satisfying $\clubsuit$ for which the Abel-Jacobi map is shown to be nontrivial using Proposition~\ref{CriterionIntroduction}. The variation of weight $1$ Hodge structures of the constructed Lagrangian families is not maximal: the general fibers of the period map of weight $1$ Hodge structures are of dimension $2$. This shows that the maximality of the variation of weight $1$ Hodge structures is a condition that cannot be dropped.

In Section~\ref{SectionMaxVar}, we shall explore under which conditions the variation of weight $1$ Hodge structures is maximal. Let $H^2(X,\mathbb Q)_{tr}$ be the orthogonal complement of $NS(X)_{\mathbb Q}$ in $H^2(X,\mathbb Q)$ with respect to the Beauville-Bogomolov-Fujiki form $q$ of $X$ (see~\cite{Beauville}) and let $b_2(X)_{tr}$ be the dimension of $H^2(X,\mathbb Q)_{tr}$ .We prove the following result (see also Proposition~\ref{PropMaxVar}):
\begin{prop}\label{PropMaxVarIntroduction}
 Consider a Lagrangian family of a hyper-Kähler manifold $X$ of dimension $2n$ given by a diagram as in (\ref{LagrangianFamily}), satisfying condition $\clubsuit$. Assume that the Mumford-Tate group of the Hodge structure $H^2(X,\mathbb Q)$ is maximal, i.e. it is the special orthogonal group of $(H^2(X,\mathbb Q)_{tr},q)$, and assume that $b_2(X)_{tr}\geq 5$.  If $h^{1,0}(L_b)$ is smaller than $2^{\lfloor\frac{b_2(X)_{tr}-3}{2}\rfloor}$, then the variation of weight $1$ Hodge structures of $p$ is maximal. 
\end{prop}

\begin{cor}
Under the same assumptions as in Proposition~\ref{PropMaxVarIntroduction}, the Abel-Jacobi map (\ref{AbelJacobiMap}) is trivial.
\end{cor}

Let $p: \mathcal L\to B$, $q: \mathcal L\to X$ be a Lagrangian family, and let
\[\pi:\mathcal A:=Alb(\mathcal L/B)\to B
\] 
be the relative Albanese variety of $p: \mathcal L\to B$. In the proof of Theorem~\ref{MainTheorem} and Proposition~\ref{PropMaxVarIntroduction}, we use a similar construction to those in~\cite{LazaSaccaVoisin, VoisinTriangle} to get a holomorphic $2$-form $\sigma_{\mathcal A}$ on $\mathcal A$. If we assume the condition $\clubsuit$,  $\pi:\mathcal A\to B$ is a Lagrangian fibration with respect to $\sigma_{\mathcal A}$(see Section~\ref{SectionLagrangianFibrations}). It is interesting to notice that, by this construction, under condition $\clubsuit$, we can translate the problem concerning Lagrangian families to a problem concerning Lagrangian fibrations. However, the total space of the Lagrangian fibration is no longer a hyper-Kähler manifold, but a completely integrable system over an open subset of the base, as introduced and studied in~\cite{DonagiMarkman}.

The organisation of the article is as follows. In Section~\ref{SectionCrierion}, we prove Proposition~\ref{CriterionIntroduction}. In Section~\ref{SectionLagrangianFibrations}, we construct a Lagrangian fibration structure on the relative Albanese variety and use it to prove Theorem~\ref{MainTheorem}. In Section~\ref{SectionMaxVar}, we discuss the condition on the maximality of the variation of Hodge structures. In Section~\ref{Examples}, we construct Lagrangian families satisfying $\clubsuit$ whose Abel-Jacobi map is nontrivial, showing that Theorem~\ref{MainTheorem} is optimal. 

\subsection*{Acknowledgement}
I would like to thank Claire Voisin for sharing her ideas with me. This paper would not be finished without her insight, instructions and encouragement. Many constructions and examples in the article are fruits of friendly and interesting discussions with her. I would also like to thank Fabrizio Anella, Peter Yi Wei and  Tim Ryan for helpful discussions.

This work was done during the preparation of my PhD thesis. I would like to thank the \emph{Institut de Mathématiques de Jussieu - Paris Rive Gauche} for marvelous research environment. The thesis is supported by the ERC Synergy Grant HyperK (Grant agreement No. 854361).

\section{A Criterion}\label{SectionCrierion}
In this section, we establish a criterion for the vanishing of the Abel-Jacobi map (\ref{AbelJacobiMap}).  The notation $j_b: L_b\to X$ as in the introduction, we prove

\begin{prop}\label{Criterion}
 Consider a Lagrangian family of hyper-Kähler manifold $X$ of dimension $2n$ given by a diagram as in (\ref{LagrangianFamily}).  \\
(a) If for general $b\in B$, the restriction map
\[j_b^*: H^{2n-1}(X,\mathbb Q)\to H^{2n-1}(L_b,\mathbb Q)
\]
is zero, then the Abel-Jacobi map (\ref{AbelJacobiMap}) is trivial.\\
(b) If condition $\clubsuit$ holds (see Proposition~\ref{CriterionIntroduction}), then the converse of (a) holds.
\end{prop}

\begin{rmk}
Since $H^{2n-1}(L_b,\mathbb Q)$ has a weight $1$ Hodge structure, by Hodge symmetry and using the fact $j_b^*: H^{2n-1}(X,\mathbb Q)\to H^{2n-1}(L_b,\mathbb Q)$ is a morphism of Hodge structures, $j_b^*: H^{2n-1}(X,\mathbb Q)\to H^{2n-1}(L_b,\mathbb Q)$ is zero, if and only if $j_b^*: H^{n-1,n}(X)\to H^{n-1,n}(L_b)$ is zero.
\end{rmk}

\begin{prf}
Let $\Psi_{\mathcal L, *, b}^{AJ}: T_{B,b}\to H^n(X,\Omega_X^{n+1})$ denote the differential of the Abel-Jacobi map $\Psi_{\mathcal L}^{AJ}$ at point $b\in B$. Let $j_{b*}: H^0(L_b,\Omega_{L_b})\to H^n(X,\Omega_X^{n+1})$ be the Gysin map, which is the Serre dual of the following composition
\begin{equation}\label{SerreDualOfTheGysinMap}
    j_b^*: H^n(X,\Omega_X^{n-1})\to H^n(L_b,\Omega_{\mathcal L|L_b}^{n-1})\to H^n(L_b,\Omega_{L_b}^{n-1}).
\end{equation}
By the above remark, the proposition follows from the following lemma and the fact that 
\[\cup\sigma_X: H^n(X,\Omega_X^{n-1})\to H^n(X,\Omega_X^{n+1})\]
is an isomorphism since $\wedge\sigma_X: \Omega_X^{n-1}\to \Omega_X^{n+1}$ is a vector bundle isomorphism.
\end{prf}
\begin{lmm}
The following diagram is commutative:
\begin{equation}\label{CommDiagLmm}
    \begin{tikzcd}
    T_{B,b}\arrow{r}{\Psi_{\mathcal L, *, b}^{AJ}}\arrow{d}{\lrcorner q^*\sigma_X}& H^n(X,\Omega_X^{n-1})\arrow{d}{\cup \sigma_X}\\
    H^0(L_b,\Omega_{L_b})\arrow{r}{j_{b*}}&    H^n(X,\Omega_X^{n+1}).
    \end{tikzcd}
\end{equation}
\end{lmm}
\begin{prf}
We are going to show that the Serre dual of the diagram (\ref{CommDiagLmm}) is commutative.

Let $L^\bullet\Omega^i_{\mathcal L|L_b}$ be the Leray filtration~\cite[Chapter 16]{Voisin} induced on the vector bundle $\Omega_{\mathcal L|L_b}^i$ by the exact sequence
\[0\to \Omega_{B, b}\otimes \mathcal O_{L_b}\to \Omega_{\mathcal L|L_b}\to \Omega_{L_b}\to 0,
\]
and defined by $L^j\Omega^i_{\mathcal L|L_b}=\Omega_{B,b}^j\wedge\Omega_{\mathcal L|L_b}^{i-j}$. Since $L_b$ is supposed to be Lagrangian, $q^*\sigma_X\in H^0(L_b, L^1\Omega_{\mathcal L|L_b}^2)$ and thus the cup product 
\[\cup q^*\sigma_X: \Omega_{\mathcal L|L_b}^{\bullet}\to \Omega_{\mathcal L|L_b}^{\bullet+2}
\]
sends $L^k\Omega_{\mathcal L|L_b}^{\bullet}$ to $L^{k+1}\Omega_{\mathcal L|L_b}^{\bullet+2}$. Denoting $\overline{q^*\sigma_X}$ the image of $q^*\sigma_X$ in $H^0(L_b, Gr_L^1\Omega_{\mathcal L|L_b})\cong H^0(L_b, \Omega_{L_b})\otimes \Omega_{B,b}$, this implies the existence of the following commutative diagram
\begin{equation}\label{CommDiagLerayFilt}
    \begin{tikzcd}
    L^1\Omega_{\mathcal L|L_b}^{n+1}=\Omega_{\mathcal L|L_b}^{n+1}\arrow{r}{}& K_{L_b}\otimes \Omega_{B,b}=Gr_L^1\Omega_{\mathcal L|L_b}^{n+1}\\
    L^0\Omega_{\mathcal L|L_b}^{n-1}=\Omega_{\mathcal L|L_b}^{n-1}\arrow{r}{}\arrow{u}{\cup q^*\sigma_X}&\Omega_{L_b}^{n-1}\arrow{u}{\cup\overline{q^*\sigma_X}}=Gr_L^0\Omega_{\mathcal L|L_b}^{n-1},
    \end{tikzcd}
\end{equation}
where $K_{L_b}$ is the canonical bundle of ${L_b}$.
Taking the $n$-th cohomology of (\ref{CommDiagLerayFilt}) and combine it with $q^*: H^n(X,\Omega_X^\bullet)\to H^n(L_b,\Omega_{\mathcal L|L_b}^\bullet)$, we get the following commutative ladder
\begin{equation}\label{CommLadder}
    \begin{tikzcd}
    H^n(X,\Omega_{X}^{n+1})\arrow{r}{q^*} & H^n(L_b,\Omega_{\mathcal L|L_b}^{n+1})\arrow{r}{}& H^n(L_b,K_{L_b}\otimes \Omega_{B,b})\cong \Omega_{B,b}\\
    H^n(X,\Omega_X^{n-1})\arrow{r}{q^*}\arrow{u}{\cup \sigma_X}& H^n(L_b,\Omega_{\mathcal L|L_b}^{n-1})\arrow{r}{}\arrow{u}{\cup \overline{q^*\sigma_X}} & H^n(L_b, \Omega_{L_b}^{n-1})\arrow{u}{\cup q^*\sigma_X}.
    \end{tikzcd}
\end{equation}

However, it is well-known~\cite[Chapter 12]{Voisin} that the composite in the fist row of the diagram (\ref{CommLadder}) coincides with the dual of $\Psi_{\mathcal L,*,b}$, while as in (\ref{SerreDualOfTheGysinMap}), the composite in the second row is $j_b^*$, so that this diagram is indeed the Serre dual of the diagram (\ref{CommDiagLmm}). This concludes the proof of the lemma.
\end{prf}

\section{Lagrangian Fibrations}\label{SectionLagrangianFibrations}
In this section, we associate to any Lagrangian family satisfying condition $\clubsuit$ a Lagrangian fibration with the help of a construction appeared in~\cite{LazaSaccaVoisin, VoisinTriangle}, and use this Lagrangian fibration to prove Theorem~\ref{MainTheorem}.

\subsection*{General Constructions}
Let (\ref{LagrangianFamily}) be a Lagrangian family of a hyper-Kähler manifold $X$ of dimension $2n$. We fix a relative polarization of $\mathcal L\to B$ given by a hyperplane section of $X$. Let \[
\pi: \mathcal A:=Alb(\mathcal L/B)\to B
\]
be the relative Albanese variety of $p: \mathcal L\to B$. 

\begin{lmm}\label{MumfordConstruction}
Let $l$ be the relative dimension of $p: \mathcal L\to B$. Then there exist an open dense subset $B_0\subset B$ and a finite covering $B_0'\to B_0$ such that, denoting $p_0':\mathcal L_0'\to B_0'$ the base change of $p$ under $B_0'\to B_0\hookrightarrow B$ and $\pi_0': \mathcal A_0'\to B_0'$ the relative Albanese variety of $p_0'$, there is a cycle $Z_0\in CH^l(\mathcal A_0'\times_{B_0'}\mathcal L_0')$ such that \[[Z_0]^*: R^0p_{0*}'\Omega_{\mathcal L_0'/B_0'}\to R^0\pi_{0*}'\Omega_{\mathcal A_0'/B_0'}\]
%Then up to passing to completion and base change of $p: \mathcal L\to B$ by a generically finite cover $B^0'\to B^0$ over an open dense subset $B^0\subset B$, there exists a cycle $Z\in CH^l(\mathcal A\times_B\mathcal L)$ such that \[[Z]^*: p_*\Omega_{\mathcal L/B}\to \pi_*\Omega_{\mathcal A/B}\]
is an isomorphism.
\end{lmm}
\begin{prf}
Let $B_0\subset B$ be the subset of regular points of $p:\mathcal L\to B$. For $b\in B_0$, let $C_b\subset L_b$ be a complete intersection curve and $J_{C_b}$ the Jacobian variety of $C_b$. By Lefschetz theorem on hyperplane sections, $j_*: J_{C_b}\to A_b:=Alb(L_b)$ is surjective. By the semi-simplicity of polarized Hodge structures, there exists a $\mathbb Q$-section $s: A_b\to J_{C_b}$ of $j_*$, i.e., there exists $N>0$ such that $j_*\circ s=N\cdot id_{A_b}$. On $J_{C_b}\times C_b$, we have the Poincaré divisor $d_b\in CH^1(J_{C_b}\times C_b)$ such that $[d_b]^*: H^1(C_b,\mathbb Q)\to H^1(J_{C_b}, \mathbb Q)$ is an isomorphism of Hodge structures. Let us consider 
\[\begin{tikzcd}
A_b\times C_b\arrow{r}{(s,id_{C_b})}\arrow{d}{(id_{A_b},j)} & J_{C_b}\times C_b\\
A_b\times L_b & 
\end{tikzcd}
\]
and define $Z_b:=(s,id_{C_b})^*(id_{A_b},j)_*(d_b)\in CH^l(A_b\times C_b)$. Then $[Z_b]^*: H^1(L_b,\mathbb Q)\to H^1(A_b,\mathbb Q)$ is given by the composition $H^1(L_b,\mathbb Q)\stackrel{j^*}{\to} H^1(C_b,\mathbb Q)\stackrel{d^*}{\to} H^1(J_{C_b},\mathbb Q)\stackrel{s^*}{\to} H^1(A_b,\mathbb Q)$, which is an isomorphism by the definition of $s$.

The cycles $Z_b$ are defined fiberwise, but standard arguments \cite[Chapter 3]{VoisinCitrouille} show that for an adequate choice of $N$ appearing in the $\mathbb Q$-section, they can be constructed in family over a smooth generically finite cover $B_0'\to B_0$. We thus get a cycle $Z_0\in CH^l(\mathcal A_0'\times_{B_0'}\mathcal L_0')$ satisfying the desired properties.
%such that $[Z_0]^*: p_*\Omega_{\mathcal L_0'/B^0'}\to \pi_*\Omega_{\mathcal A_0'/B^0'}$ is an isomorphism. We choose a smooth completion $p':\mathcal L'\to B'$ of $\mathcal L_0'\to B^0'$ and let $\pi': \mathcal A'\to B'$ be the relative Albanise variety of $p'$. Then $Z_0$ extends by taking the Zariski closure to a cycle $Z\in CH^l(\mathcal A'\times B'\mathcal L')$ with the desired property.
\end{prf}

For the sake of simplicity, we shall note $B_0$, $\mathcal L_0$ and $\mathcal A_0$ instead of $B_0'$, $\mathcal L_0'$ and $\mathcal A_0'$. We define a holomorphic $2$-form $\sigma_{\mathcal A_0}$ on $\mathcal A_0$ by setting
\begin{equation}\label{2FormOnA}
    \sigma_{\mathcal A_0}:=[Z_0]^*q_0^*\sigma_X,
\end{equation}
where $q_0:\mathcal L_0\to X$ is the natural map. 
\begin{prop}\label{PropertiesOfThe2Form}
(a) The $2$-form $\sigma_{\mathcal A_0}$ is closed.\\
(b) $\sigma_{\mathcal A_0}$ vanishes on fibers of $\pi_0:\mathcal A_0\to B_0$.\\
(c) The composite morphism $(\lrcorner q_0^*\sigma_X): T_{B_0}\to R^0p_{0*}\Omega_{\mathcal L_0/B_0}\cong R^0\pi_{0*}\Omega_{\mathcal A_0/B_0}$ is given by the contraction $\lrcorner\sigma_{\mathcal A_0}$.
\end{prop}
\begin{prf}
(a) Let $Z_q:=(id, q_0)_*Z_0\in CH(\mathcal A_0\times X)$. Then by the projection formula, $\sigma_{\mathcal A_0}=[Z_q]^*\sigma_X$. Let $\mathcal A'$ be a completion of $\mathcal A_0$. Then $Z_q$ extends to a cycle $\bar Z_q$ of $\mathcal A'\times X$. $\sigma_{\mathcal A_0}$ extends to a $2$-form $\sigma_{\mathcal A'}:=[\bar Z_q]^*\sigma_X$ which is automatically closed since $\mathcal A'$ is projective. Thus, $\sigma_{\mathcal A_0}=\sigma_{\mathcal A'|\mathcal A_0}$ is also closed.

(b) Since $Z_0$ is a cycle in $\mathcal A_0\times_{B_0}\mathcal L_0\subset \mathcal A_0\times \mathcal L_0$, $[Z_0]^*: H^*(\mathcal L_0)\to H^*(\mathcal A_0)$ preserves the Leray filtrations on both sides. Therefore, 
$\sigma_{\mathcal A_0}\in H^0(\mathcal A_0, \pi_0^*\Omega_{B_0}\wedge \Omega_{\mathcal A_0})\subset H^0(\mathcal A_0, \Omega_{\mathcal A_0}^2)$ since $q_0^*\sigma_X\in H^0(\mathcal L_0, p_0^*\Omega_{B_0}\wedge \Omega_{\mathcal L_0})$ by the definition of Lagrangian families. Therefore, $\sigma_{\mathcal A_0}$ vanishes on the fibers of $\pi_0:\mathcal A_0\to B_0$.

(c) By Lemma~\ref{MumfordConstruction}, $[Z_0]^*$ induces an isomorphism $H^0(B_0, \Omega_{B_0}\otimes p_{0*}\Omega_{\mathcal L_0/B_0})\to H^0(B_0, \Omega_{B_0}\otimes \pi_{0*}\Omega_{\mathcal A_0/B_0})$ which sends $\lrcorner q_0^*\sigma_X$ to $\lrcorner\sigma_{\mathcal A_0}$.
\end{prf}

By (b) and (c) of the above Proposition, we get the following commutative diagram:
\begin{equation}\label{MainCommDiag}
    \begin{tikzcd}
    0 \arrow{r}{} & T_{\mathcal A_0/B_0} \arrow{r}{}\arrow{d}{(\lrcorner q_0^*\sigma_X)^*} & T_{\mathcal A_0} \arrow{r}{} \arrow{d}{\lrcorner\sigma_{\mathcal A_0}} & \pi_0^*T_{B_0} \arrow{r}{}\arrow{d}{\lrcorner q_0^*\sigma_X} & 0\\
    0 \arrow{r}{}& \pi_0^*\Omega_{B_0} \arrow{r}{} & \Omega_{\mathcal A_0}\arrow{r}{} &\Omega_{\mathcal A_0/B_0}\arrow{r}{} & 0.
    \end{tikzcd}
\end{equation}

\begin{lmm}
If condition $\clubsuit$ (see Proposition~\ref{CriterionIntroduction}) holds for all $b\in B_0$, then $\sigma_{\mathcal A_0}$ is non degenerate on $\mathcal A_0$.
\end{lmm}
\begin{prf}
If condition $\clubsuit$ holds, then $(\lrcorner q_0^*\sigma_X): \pi_0^*T_{B_0}\to \Omega_{\mathcal A_0/B_0}$ is an isomorphism.
By the commutativity of (\ref{MainCommDiag}) and the five lemma, $\lrcorner\sigma_{\mathcal A_0}: T_{\mathcal A_0}\to \Omega_{\mathcal A_0}$ is an isomorphism, which means that $\sigma_{\mathcal A_0}$ is nowhere degenerate.
\end{prf}

\subsection*{Symmetry}
Let (\ref{LagrangianFamily}) be a Lagrangian family of a hyper-Kähler manifold $X$ of dimension $2n$. We fix a relative polarization of $\mathcal L\to B$ given by a hyperplane section of $X$. Let $b\in B$ be a general point. The infinitesimal variation of Hodge structures on degree $1$ cohomology of the fibers of $p: \mathcal L\to B$ at $b$ is given by (see~\cite[Chapters 10, 17]{Voisin})
\[\bar\nabla: T_{B,b}\to Hom(H^0(L_b,\Omega_{L_b}), H^1(L_b,\mathcal O_{L_b})).
\]
Composed with the map $\lrcorner q^*\sigma_X: T_{B,b}\to H^0(L_b,\Omega_{L_b})$, $\bar\nabla$ induces a bilinear map
\begin{equation}
    \begin{array}{cccc}
       S:  & T_{B,b}\times T_{B,b} & \to & H^1(L_b,\mathcal O_{L_b})  \\
         & (u,v) & \mapsto & \bar\nabla_u(v\lrcorner q^*\sigma_X).
    \end{array}
\end{equation}
\begin{prop}\label{SymmetryProp}
The bilinear map $S$ is symmetric in the sense that $S(u,v)=S(v,u)$ for any $u,v\in T_{b,B}$.
\end{prop}
\begin{prf}
By Griffiths' transversality~\cite[Chapter 17]{Voisin}, $S(u,v)=\rho(u)\lrcorner(v\lrcorner q^*\sigma_X)$, where $\rho: T_{B,b}\to H^1(L_b, T_{L_b})$ is the Kodaira-Spencer map. Therefore, we need to show that the following diagram is commutative
\begin{equation}\label{CommDiagSymm}
  \begin{tikzcd}
   T_{B,b}\otimes T_{B,b} \arrow{r}{id\otimes \rho}\arrow{d}{\lrcorner q^*\sigma_X\otimes id} & T_{B,b}\otimes H^1(L_b, T_{L_b}) \arrow{d}{\beta}\\
   H^0(L_b,\Omega_{L_b})\otimes T_{B,b}\arrow{r}{\alpha} & H^1(L_b,\mathcal O_{L_b}),
  \end{tikzcd}
\end{equation}
where $\alpha(\omega, v)=\omega\lrcorner \rho(v)$ and $\beta(u,\chi)=(q^*\sigma_X\lrcorner \chi)$. 

To see the commutativity of (\ref{CommDiagSymm}), restrict the commutative ladder (\ref{MainCommDiag}) to the point $b\in B$ and apply the cohomology on $L_b$, then the commutativity of (\ref{MainCommDiag}) implies the commutativity of (\ref{CommDiagSymm}). Indeed, (\ref{CommDiagSymm}) is the connection map of the cohomology of (\ref{MainCommDiag}) tensored by $T_{B,b}$.
\end{prf}

\begin{rmk}
When condition $\clubsuit$ is satisfied, the symmetry of $S$ comes from the completely integrable system structure on $(\mathcal A_0, \sigma_{\mathcal A_0})$. What we proved is in fact the symmetry of 
\begin{equation}
    \begin{array}{cccc}
       S':  & T_{B,b}\times T_{B,b} & \to & H^1(A_b,\mathcal O_{A_b})  \\
         & (u,v) & \mapsto & \bar\nabla_u(v\lrcorner \sigma_{\mathcal A_0}).
    \end{array}
\end{equation}
Fixing a relative polarisation on $\mathcal A_0\to B_0$, we have natural isomorphisms (always under $\clubsuit$): $H^1(A_b,\mathcal O_{A_b})\cong H^0(A_b, \Omega_{A_b})^*\cong T_{B,b}^*$ and we can thus view $S'$ as an element in $T_{B,b}^*\otimes T_{B,b}^*\otimes T_{B,b}^*$. If this relative polarisation is principal, Donagi and Markman proved in~\cite{DonagiMarkman} that $S'$ lies in $Sym^3T_{B,b}^*$. This result is called ``weak cubic condition'' in~\cite{DonagiMarkman}.
\end{rmk}

Now we are ready to prove the main theorem.
\begin{proof}[Proof of Theorem~\ref{MainTheorem}]
Under the assumptions of Theorem~\ref{MainTheorem}, assume by contradiction that the Abel-Jacobi map (\ref{AbelJacobiMap}) is not constant. In what follows, we fix a relative polarization on $p_0:\mathcal L_0\to B_0$ induced from a hyperplane section of $X$, so that $R^{2n-1}p_{0*}\mathbb Q\cong R^1p_{0*}\mathbb Q$. By Proposition~\ref{CriterionIntroduction}, the morphism
\[j^*: H^{2n-1}(X,\mathbb Q)\to R^{2n-1}p_{0*}\mathbb Q\cong R^1p_{0*}\mathbb Q
\]
 of variations of Hodge structures on an open subset $B_0\subset B$ containing $b$ is not zero. Hence, there is a non-zero locally constant sub-variation of Hodge strutures $I:=\mathrm{Im}j^*\subset R^1p_{0*}\mathbb Q$. Since $I$ is locally constant, for any $\omega\in I^{1,0}_b$ and $u\in T_{B,b}$, $\nabla_u(\omega)=0$. Recall that $\clubsuit$ means that $\lrcorner q^*\sigma_X: T_{B,b}\to H^0(L_b,\Omega_{L_b})$ is bijective. Let $F:=(\lrcorner q^*\sigma_X)^{-1}(I^{1,0})\subset T_{B,b}$. Then by the symmetry of $S$ given by Proposition~\ref{SymmetryProp}, $F$ lies in the kernel of $\bar\nabla$, which contradicts our assuption that the variation of Hodge structures is maximal.

\end{proof}

\section{Maximal Variations}\label{SectionMaxVar}
In this section, we study under what conditions could the variation of Hodge structures of a Lagrangian family be maximal. Consider a Lagrangian family of a hyper-Kähler manifold $X$ of dimension $2n$ satisfying the condition $\clubsuit$ given by the diagram as in (\ref{LagrangianFamily}). Let $U\subset B$ be a simply connected open subset of $B_0\subset B$ and let
\begin{equation}\label{ApplicationModules}
    \begin{array}{cccc}
         \mathcal P: & U &\to & Gr(h^{1,0}(L), H^2(L,\mathbb C))\\
          & b & \mapsto & H^{1,0}(L_b)\subset H^1(L_b,\mathbb C)\cong H^1(L,
          \mathbb C),
    \end{array}
\end{equation}
be the local period map of the Lagrangian family. %$\mathcal D$ can also be understood as a universal space of abelian varieties given by the relative Albanese variety $\pi: \mathcal A\to B$ of $p: \mathcal L\to B$.

In what follows, we are going to use a universal property of the Kuga-Satake construction proved in~\cite{VoisinvGeemen}.
\begin{thm}[\cite{VoisinvGeemen}]\label{universalKS}
Let $(H^2,q)$ be a polarized Hodge structure of hyper-Kähler type of dimension $\geq 5$. Assume that the Mumford–Tate group of the Hodge structure on $H^2$ is maximal, namely the special orthogonal group of $(H^2,q)$. Let $H$ be a simple effective weight-$1$ Hodge structure, such that there exists an injective morphism of Hodge structures of bidegree $(-1,-1)$
\[H^2\hookrightarrow Hom(H, A)
\]
for some weight-$1$ Hodge structure $A$. Then $H$ is a direct summand of the Kuga–Satake Hodge
structure $H^1_{KS}(H^2,q)$. In particular, $\dim H\geq 2^{\lfloor\frac{\dim H^2-1}{2}\rfloor}$.
\end{thm}

With the same notations as in the introduction, we prove

\begin{prop}\label{PropMaxVar}
Assume that the Mumford-Tate group of the Hodge structure $H^2(X,\mathbb Q)$ is maximal, i.e. it is the special orthogonal group of $(H^2(X,\mathbb Q)_{tr},q)$ and assume $b_2(X)_{tr}\geq 5$. If the dimension of $H^{0,1}(L_b)$ is smaller than $2^{\lfloor\frac{b_2(X)_{tr}-3}{2}\rfloor}$ for a general fiber $L_b$ of $p: \mathcal L\to B$, then the variation of weight $1$ Hodge structure of $p$ is maximal.
\end{prop}

\begin{prf}
We use the same argument as in~\cite{VoisinvGeemen} where similar results were proved for Lagrangian fibrations. Assuming that the period map (\ref{ApplicationModules}) is not generically an immersion, we are going to prove that $\dim H^{0,1}(L_b)\geq 2^{\lfloor\frac{b_2(X)_{tr}-3}{2}\rfloor}$. By assumption, the nonempty general fibers of $\mathcal P$ are of dimension $\geq 1$. Let $b\in U$ be a general point and let $B_b$ the fiber of $\mathcal P$ passing through $b$. Let $U_b=B_b\cap U$. Then the fibers of $\pi|_{U_b}: \mathcal A_{U_b}\to U_b$ are isomorphic with each other.  Thus, up to a base change by a finite covering of $U_b$, we may assume $\pi|_{U_b}: \mathcal A_{U_b}\to U_b$ is trivial, i.e., $\mathcal A_{U_b}=U_b\times A_b$. Let $\pi_{F_b}:\mathcal A_{F_b}\to F_b$ be a smooth completion of $\pi|_{U_b}$, then $\mathcal A_{F_b}$ is birational to $F_b\times A_b$, which gives a morphism $H^2(\mathcal A_{F_b})\to H^2(F_b\times A)$. Recall by Lemma~\ref{MumfordConstruction}, we get a morphism $[Z]^*: H^2(X)\to H^2(\mathcal A)$ that sends $\sigma_X$ to a holomorphic $2$-form which is non-degenerate on $\mathcal A_U$. Finally, the rational map $\mathcal A_{F_b}\dashrightarrow \mathcal A$ induces $H^2(\mathcal A)\to H^2(\mathcal A_{F_b})$. Compositing all these maps, we get a morphism
\begin{equation}
\alpha: H^2(X)_{tr}\hookrightarrow H^2(X)\to H^2(\mathcal A)\to H^2(\mathcal A_{F_b})\to H^2(F_b\times A_b)\to H^1(F_b)\otimes H^1(A_b),
\end{equation}
where the last map is given by the projection in the Künneth decomposition. 
\begin{lmm}
$\alpha: H^2(X)_{tr}\to H^1(F_b)\otimes H^1(A_b)$ is injective.
\end{lmm}
\begin{prf}
Since $h^{2,0}(X)=1$ and $H^{2,0}(X)$ is orthogonal to $NS(X)$ with respect to the Beauville-Bogomolov-Fujiki form, $H^2(X)_{tr}$ is a simple Hodge structure. Therefore, to show the injectivity of $\alpha$ it suffices to show that $\alpha$ is not zero. We claim that $\alpha(\sigma_X)\neq 0$. Indeed, Since $A_b$ is Lagrangian with respect to $\sigma_{\mathcal A}$ (Proposition~\ref{PropertiesOfThe2Form} (b)), in the Künneth's decomposition of $H^2(A_b\times F_b)$, the image of $\sigma_X$ in $H^0(F_b)\otimes H^2(A_b)$ is zero. If furthermore $\alpha(\sigma_X)=0$ in $H^1(F_b)\otimes H^1(A_b)$, then the image of $\sigma_X$ on $F_b\times A_b$ comes from a $2$-form on $F_b$, which has rank $\leq \dim F_b$. Therefore, the rank of $\sigma_{\mathcal A}$ has rank $\leq \dim F_b$ on $\mathcal A_{U_b}$. On the other hand, the codimension of $\mathcal A_{U_b}$ in $\mathcal A_U$ is $\dim B-\dim F_b$, and thus the non-degeneration of $\sigma_{\mathcal A_U}$ implies that $\sigma_{\mathcal A}$ has rank $\geq 2\dim F_b$ on $\mathcal A_{U_b}$. This is a contradiction since we are assuming $\dim F_b\geq 1$.
\end{prf}
We are now in the position to use the universal property of the Kuga-Satake construction (see Theorem~\ref{universalKS} above). Since $\alpha: H^2(X)_{tr}\to H^1(F_b)\otimes H^1(A_b)$ is nonzero, there is at least one simple direct factor $A$ of $A_b$ such that $H^2(X)_{tr}\to H^1(F_b)\otimes H^1(A)$ is nonzero thus injective. Taking $H^2$ as $H^2(X)_{tr}$, we conclude by Theorem~\ref{universalKS} that 
\[\dim H^{0,1}(L_b)=\dim A_b\geq \dim A\geq \frac12\times 2^{\lfloor \frac{\dim H^2(X)_{tr}-1}{2}\rfloor}=2^{\lfloor\frac{b_2(X)_{tr}-3}{2}\rfloor},\]
as desired.
\end{prf}

\section{Example of a Lagrangian Family with Nontrivial Abel-Jacobi Map}\label{Examples}
%\subsection{Generalized Kummer Varieties}
Recall the construction of generalized Kummer varieties introduced in~\cite{Beauville}. Let $A$ be an abelian surface and $A^{[n+1]}$ the Hilbert scheme of length $n+1$ subschemes of $A$. Let $\mathrm{alb}: A^{[n+1]}\to A$ be the composition of The Hilbert-Chow morphism and the summation map
\[A^{[n+1]}\to A^{(n+1)}\to A.\]
Then $alb$ is an isotrivial fibration.
The generalized Kummer variety $K_n(A)$ is defined to be the fiber of $alb$ over $0\in A$. As is shown in~\cite{Beauville}, $K_n(A)$ is a hyper-Kähler manifold of dimension $2n$.

In this section, we are going to construct Lagrangian families of $X:=K_n(A)$ for $n\geq 2$, satisfying condition $\clubsuit$ and whose Abel-Jacobi map is \emph{not} trivial. 

For any $x\in A$, one defines a subvariety $Z_x$ of $K_n(A)$ consisting of Artinian subschemes of $A$ of length $n+1$ supported on $x$ and $-nx$, with multiplicities $n$ and $1$, respectively. By \cite[Proposition VI.1.1]{Briancon}, $Z_x$ is a \emph{rational} variety of dimension $n-1$ if $x$ is \emph{not} an $(n+1)$-torsion point. Let $Z=\bigcup_{x\in A} Z_x$ and let $\pi: Z\to A$ send elements in $Z_x$ to $x$. For any curve $C\subset A$, define $Z_C=\bigcup_{x\in C} Z_x$. 

Now let $B$ be a connected open subset of the Hilbert scheme of deformations of a smooth curve $C\subset A$ and $\mathcal C\to B$ the corresponding family.
\begin{lmm}
$\{Z_C\}_{C\in B}$ is a Lagrangian family of $K_n(A)$ satifying condition $\clubsuit$.
\end{lmm}
\begin{prf}
Since for general $C$, $Z_C$ is a fibration over a curve $C$ whose general fibers are rational, any holomorphic $2$-form on $Z_C$ is $0$. Furthermore, $\dim Z_C=n=\dim K_n(A)/2$. These imply that $\{Z_C\}_{C\in B}$ is a Lagrangian family. We now show that this family satisfies condition $\clubsuit$. Denoting $\mathcal L$ the total space of the family $\{Z_C\}_{C\in B}$ and $L$ a general fiber, and using as before the following notation
\[
\begin{tikzcd}
    \mathcal L\arrow{r}{q}\arrow{d}{p} & X\\
    B &
    \end{tikzcd},
    \]
    we need to show that $\lrcorner q^*\sigma_{K_n(A)}: H^0(L,N_{L/Z})=H^0(L, N_{L/\mathcal L})\to H^0(L,\Omega_L)$ is an isomorphism. Since the general fibers of $\pi$ are rational, $q^*\sigma_{K_n(A)}=\tilde\pi^*\sigma_A$, where $\sigma_A$ is the unique (up to coefficients) holomorphic $2$-form on $A$. Therefore, we can conclude by the commutativity of the following diagram
\[\begin{tikzcd}
 H^0(C,N_{C/A})\arrow{r}{\lrcorner\sigma_A}\arrow{d}{\pi^*} & H^0(C,\Omega_C)\arrow{d}{\pi^*}\\
 H^0(L,N_{L/ Z})\arrow{r}{\lrcorner\pi^*\sigma_A}& H^0(L,\Omega_L)
\end{tikzcd}
\]
noting that the two vertical arrows are isomorphims since the fibers of $\pi$ are rational, and that $\lrcorner \sigma_A: H^0(C,N_{C/A})\to H^0(C,\Omega_C)$ is an isomorphism since $\sigma_A$ is nondegenerate.
\end{prf}

\begin{prop}
The Abel-Jacobi map of the Lagrangian family $\{Z_C\}_{C\in B}$ is \emph{not} trivial.
\end{prop}
\begin{prf}
Let $i: C\hookrightarrow A$ be a general curve in the family $\mathcal C\to B$. By Proposition~\ref{Criterion}(a), it suffices to show that the restriction map $H^{2n-1}(X,\mathbb Q)\to H^{2n-1}(Z_C,\mathbb Q)$ is nonzero. 

Define an injective morphism
\[\begin{array}{cccc}
     \beta:& A  \times  A & \hookrightarrow & A^{(n+1)}\\
     & (x, y) & \mapsto & n\{x\}+\{y\},
\end{array}
\]
where we use the notation $\{x\}\in \mathcal Z_0(A)$ the $0$-cycle of the point $x\in A$. Consider the following pull-back diagram defining a subvariety $Z'\subset A^{[n+1]} $
\[\begin{tikzcd}
Z'\arrow[r, hookrightarrow, "\alpha"]\arrow{d}{\pi'} &  A^{[n+1]} \arrow{d}{c}\\
A\times A \arrow[r, hookrightarrow, "\beta"] & A^{(n+1)}
\end{tikzcd},
\]
where $c:A^{[n+1]}\to A^{(n+1)}$ is the Hilbert-Chow morphism.
Then $Z = Z'\cap K_n(A)\subset A^{[n+1]}$. We have the following commutative diagram where all three squares are pull-back diagrams
\[
\begin{tikzcd}
Z\arrow[rr, hookrightarrow]\arrow[dd, "\pi"]\arrow[rd, hookrightarrow] && X= K_n(A) \arrow[rd, hookrightarrow] \\
& Z'\arrow[rr, hookrightarrow, "\alpha"]\arrow[dd,"\pi'"] &&  A^{[n+1]} \arrow[dd, "c"]\arrow[rdd, "\mathrm{alb}"]\\
A \arrow[rd, hookrightarrow, "f"]  \\
&A\times A \arrow[rr, hookrightarrow, "\beta"] && A^{(n+1)}\arrow[r, "\sum"] & A
\end{tikzcd},
\]
Here $f: A\to A\times A$ defined by $x\mapsto (x, -nx)$ is the fiber over $0\in A$ of the trivial fibration $\sum\circ\beta: A\times A\to A$. 

By~\cite[Corollary 5.1.5]{deCataldoMigliorini}, $[Z']^*: H^{2n-1}(A^{[n+1]}, \mathbb Q)\to H^1(A\times A, \mathbb Q)$ is surjective. Furthermore, the restriction map $f^*: H^1(A\times A, \mathbb Q)\to H^1(A, \mathbb Q)$ is surjective since $f$ is the fiber of a trivial fibration. These imply that $[Z]^*: H^{2n-1}(X, \mathbb Q)\to H^1(A, \mathbb Q)$ is surjective. Finally, since the restriction map $i^*: H^1(A, \mathbb Q)\to H^1(C, \mathbb Q)$ is injective by Lefschetz hyperplane theorem, the composition map $i^*\circ [Z]^*: H^{2n-1}(X, \mathbb Q)\to H^1(C, \mathbb Q)$ is nonzero. This implies that the restriction map $H^{2n-1}(X, \mathbb Q)\to H^{2n-1}(Z_C, \mathbb Q)$ is nonzero, as desired.
\end{prf}

 ~\newline
 
Sorbonne Université and Université de Paris, CNRS, IMJ-PRG, F-75005 Paris, France.

\textit{Email adress}: \verb|chenyu.bai@imj-prg.fr|.

\end{document}